\documentclass[a4paper]{article}
\usepackage{fullpage}
\usepackage{tikz}

\usepackage[pdftex,breaklinks,colorlinks,
linkcolor=blue,
citecolor=blue,
urlcolor=blue]{hyperref}

\usepackage{enumitem}
\usepackage{amsmath,amssymb,amsfonts,amsthm}

\newtheorem{theorem}{Theorem}

\newtheorem{lemma}[theorem]{Lemma}

\theoremstyle{definition}
\newtheorem{definition}[theorem]{Definition}
\newtheorem{remark}[theorem]{Remark}

\author{Pablo Romero\footnote{Facultad de Ingenier\'ia, Universidad de la Rep\'ublica, Montevideo, Uruguay. E-mail address: \texttt{promero@fing.edu.uy}}
\footnote{Facultad de Ciencias Exactas y Naturales, Universidad de Buenos Aires, Ciudad Universitaria. Av. Int. Güiraldes 2160. Buenos Aires, Argentina.}}

\date{}

\makeatletter%
\begin{document}

\title{Existence, uniqueness and construction\\ of locally most reliable two-terminal graphs}
\maketitle

\begin{abstract}\let\thefootnote\relax
A two-terminal graph is a graph $G$ equipped with $2$ vertices in $V(G)$ called terminals. 
Let $T_{n,m}$ be the set of two-terminal graphs on $n$ vertices and $m$ edges. Let $G$ be in $T_{n,m}$ and let $p$ be in $[0,1]$. 
The two-terminal reliability of $G$ at $p$, denoted $R_G(p)$, is the probability that $G$ has a path joining its terminals after each of its edges is independently removed with probability $1-p$. We say $G$ is a \emph{locally most reliable two-terminal graph} (LMRTTG) if for each $H$ in $T_{n,m}$ there exists a positive real number $\delta$ such that for every $p$ in $(0,\delta)$ it holds that $R_G(p) \geq R_H(p)$. 

It is simple to prove that there exists a unique LMRTTG in $T_{n,m}$ when $n\geq 4$ and $5\leq m \leq 2n-3$. Gong and Lin~[Discrete Appl. Math. 356 (2024), 393–402] further proved that there exists a unique LMRTTG in  $T_{n,m}$ when $n\geq 6$ and $2n-3 \leq m \leq \binom{n}{2}$, except for some pairs of integers $n$ and $m$ which satisfy that $n \geq 6$ and $\frac{1}{2}\binom{n-2}{2}-\frac{n-2}{2}\leq m-(2n-3) \leq \frac{1}{2}\binom{n-2}{2}+\frac{n-2}{2}$. All cases unresolved in earlier works are covered here. In this article it is proved that in each set $T_{n,m}$ such that $n\geq 4$ and $5\leq m \leq \binom{n}{2}$ 
there exists a unique LMRTTG, called $G_{n,m}$. A construction of $G_{n,m}$ is also given. 
\end{abstract}

\renewcommand{\labelitemi}{--}

\section{Introduction}\label{section:intro}
A two-terminal graph is a graph $G$ equipped with $2$ vertices in $V(G)$ called terminals. Let $T_{n,m}$ be the set consisting of all two-terminal graphs on $n$ vertices and $m$ edges. Let $G$ be in $T_{n,m}$ and let $p$ be in $[0,1]$. The two-terminal reliability of $G$ at $p$, denoted $R_G(p)$, is the probability that $G$ has a path joining its terminals after each of its edges is independently removed with probability $1-p$. We say $G$ is a \emph{uniformly most reliable two-terminal graph} if $R_G(p)\geq R_H(p)$ for each $H$ in $T_{n,m}$ and for every $p$ in $[0,1]$. We say $G$ is a \emph{locally most reliable two-terminal graph} (LMRTTG) if for each $H$ in $T_{n,m}$ there exists a positive real number $\delta$ such $R_G(p) \geq R_H(p)$ for each $H$ in $T_{n,m}$ and for every $p$ in $(0,\delta)$. 

Independent works of Bertrand et al.~\cite{Bertrand-2018} and Xie et al.~\cite{Xie-2021} conclude that there is no uniformly most reliable two-terminal graph on $n$ vertices and $m$ edges when  $n\geq 11$ and either $20 \leq m \leq 3n-9$ or $3n-5 \leq m \leq \binom{n}{2}-2$. Considering that for most cases uniformly most reliable two-terminal graphs do not exist, effort was carried out to characterize LMRTTGs. Bertrand et al.~\cite{Bertrand-2018} determined the only LMRTTG in each $T_{n.m}$  when $n\geq 4$ and $5\leq m \leq 3n-6$. 
More recently, Gong and Lin~\cite{Gong-2024} succeeded to determine the only LMRTTG in  $T_{n,m}$ when $n\geq 6$, $2n-3\leq m \leq \binom{n}{2}$ and further, the first Zagreb index of the graphs $C_{n-2,m-2n+3}^1$ and $S_{n-2,m-2n+3}^1$ are different, where 
$C_{n,m}^1$ and $S_{n,m}^1$ stand for the quasi-complete and quasi-star graphs, respectively. 

The gap unresolved in earlier works is filled here. This article is organized as follows. Section~\ref{section:concepts} includes graph theoretic concepts that are used throughout the document. Section~\ref{section:background} presents the concept of LMRTTGs, a well known methodology to construct such two-terminal graphs, and the progress in the determination of LMRTTGs. The main theorem of this article is presented in Section~\ref{section:main}. This theorem asserts that there exists a unique LMRTTG in $T_{n,m}$ 
when $n\geq 4$ and $5\leq m \leq \binom{n}{2}$, that is called $G_{n,m}$. The two-terminal graph $G_{n,m}$ is explicitly constructed. 

\section{Basic concepts}\label{section:concepts}
All graphs in this work are finite and undirected, with no loops nor multiple edges. Let $G$ be a graph. The complement of $G$ is denoted $\overline{G}$. 
The vertex set and edge set of $G$ are denoted $V(G)$ and $E(G)$, respectively. A subgraph $H$ of $G$ is a graph such that $V(H) \subseteq V(G)$ and $E(H) \subseteq E(G)$. A spanning graph of $G$ is a subgraph $H$ of $G$ such that $V(H)=V(G)$. 
Two vertices $v_i$ and $v_j$ in $G$ are  adjacent when $v_iv_j \in E(G)$. A vertex $v$ in $G$ is universal (isolated) when $v$ is adjacent to all (resp. none) of the vertices in $V(G)-v$. A threshold graph is a graph that can be constructed from $K_1$ by repeatedly adding an isolated vertex or a universal vertex. Let $P_n$, $C_n$, and $K_n$ be the path, the cycle, and the complete graph on $n$ vertices, respectively. The graph $C_3$ is sometimes called the \emph{triangle}. The number of triangles, $3$-paths, and $4$-paths in $G$ are denoted $k_3(G)$, $p_3(G)$, and $p_4(G)$, respectively. For each $v$ in $V(G)$ we denote $d_G(v)$ its degree in $G$; we will denote $d(v)$ instead of $d_G(v)$ when the graph $G$ is clear from context. The first Zagreb index of $G$, denoted $M_1(G)$, equals $\sum_{v \in V(G)}d_G(v)^2$. The second Zagreb index of $G$, denoted $M_2(G)$, equals $\sum_{v_iv_j \in E(G)}d_G(v_i)d_G(v_j)$.  
The union between two graphs $G$ and $H$ is a graph whose vertex set is $V(G) \cup V(H)$ and whose edge set is $E(G) \cup E(H)$. For each positive integer $s$ we denote $sG$ the graph consisting of a union of $s$ disjoint copies of the graph $G$. The join $G\vee H$ between the graphs $G$ and $H$ arises from $G\cup H$ by adding precisely one edge joining every vertex of $G$ to every vertex of $H$.

\section{Related work}\label{section:background}
Let $G$ be in $T_{n,m}$ and let $p$ be in $[0,1]$. 
The two-terminal reliability of $G$ at $p$, denoted $R_G(p)$, is the probability that $G$ has a path joining its terminal after each of its edges is independently removed with probability $1-p$. For each $i$ in $\{1,2\ldots,m\}$, we define $N_i(G)$ as the number of spanning subgraphs in $G$ composed by $i$ edges that have some path joining the terminals in $G$. Clearly,
\begin{equation*}
R_G(p) = \sum_{i=1}^{m}N_i(G)p^i(1-p)^{n-i}.    
\end{equation*}

The following concept is the object under study in this article.
\begin{definition}
A two-terminal graph $G$ in $T_{n,m}$ is a \emph{locally most reliable two-terminal graph} (LMRTTG) if for each $H$ in $T_{n,m}$ there exists a positive real number $\delta$ such that for every $p$ in $(0,\delta)$ it holds that $R_G(p)\geq R_H(p)$.  
\end{definition}

Lemma~\ref{lemma:local} can be proved using elementary calculus. 
\begin{lemma}[Observation 2.1 in~\cite{Brown-2014}]\label{lemma:local}
Let $G$ and $H$ be a pair of two-terminal graphs in $T_{n,m}$. If there exists $j$ in 
$\{1,2,\ldots,m\}$ such that $N_i(G)=N_i(H)$ for all $i\in\{1,2\ldots,j-1\}$ and $N_j(G)>N_j(H)$, then there exists $\delta>0$ such that $R_G(p)>R_H(p)$ for every $p$ in $(0,\delta)$.
\end{lemma}

For each $G$ in $T_{n,m}$ we define $N(G)$ as $(N_1(G),N_2(G),\ldots,N_m(G))$. 
By Lemma~\ref{lemma:local}, if $G$ is any LMRTTG in $T_{n,m}$ then $N(G)$ is maximum  
in $T_{n,m}$ under the lexicographic order. In fact, define $T_{n,m}^{(0)}$ as $T_{n,m}$. For each $i\in \{1,2,\ldots,m\}$, define $T_{n,m}^{(i)}$ as the set consisting of all two-terminal graphs in $T_{n,m}^{(i-1)}$ that maximize $N_i(G)$, i.e., 
\begin{equation*}
T_{n,m}^{(i)} = \{G: G\in T_{n,m}^{(i-1)}, \, \, N_i(G)\geq N_i(H) \text{ for each } H \text{ in } T_{n,m}^{(i-1)}\}.    
\end{equation*}

The following result is an application of Lemma~\ref{lemma:local}.
\begin{lemma}[Lemma 6 in~\cite{2025-Romero-TCS}]\label{lemma:set}
For each pair of integers $n$ and $m$ such that $T_{n,m}$ is not empty, the set consisting of all LMRTTGs in $T_{n,m}$ is not empty and equals $T_{n,m}^{(m)}$.    
\end{lemma}

In this work we will use Remark~\ref{remark:singleton}, which is a corollary of Lemma~\ref{lemma:set}.
\begin{remark}\label{remark:singleton}
If there exist a two-terminal graph $G$ in $T_{n,m}$ and some integer $j$ in $\{1,2,\ldots,m\}$ such that $T_{n,m}^{(j)}=\{G\}$, then $G$ is the only LMRTTG in $T_{n,m}$.
\end{remark}

Bertrand et al.~\cite{Bertrand-2018} found the only LMRTTG in $T_{n,m}$ when $n\geq 4$ and $5\leq m \leq 3n-6$. For our purposes, we present a statement that appears in~\cite{Bertrand-2018} which determines the only LMRTTG in $T_{n,m}$ when $n\geq 4$ and $5\leq m\leq 2n-3$.
\begin{definition}\label{definition:Hnm}
Let $n$ and $m$ be integers such that $n\geq 4$ and $5\leq m \leq 2n-3$. The two-terminal graph $G_{n,m}$ has vertex set $\{s,t\}\cup \{v_1,v_2,\ldots,v_{n-2}\}$ where $s$ and $t$ denote its terminal vertices, and its edge set $E(G_{n,m})$ is defined as follows,
\begin{enumerate}[label=(\roman*)]
\item if $m$ is odd then 
$E(G_{n,m})=\{st\} \cup \{sv_i,v_it, i\in \{1,2,\ldots,\frac{m-1}{2}\}\}$; 
\item if $m$ is even then $E(G_{n,m})=\{st\} \cup \{sv_i,v_it, i\in \{1,2,\ldots,\frac{m-2}{2}\}\} \cup \{v_1v_2\}$.
\end{enumerate}
\end{definition}

\begin{lemma}[Proposition 2.3 in~\cite{Bertrand-2018}]\label{lemma:onlyH}
Let $n$ and $m$ be integers such that $n\geq 4$ and $5\leq m \leq 2n-3$. The only LMRTTG in $T_{n,m}$ is $G_{n,m}$.
\end{lemma}

Let $\mathcal{G}_{n,m}$ be the set of simple graphs on $n$ vertices and $m$ edges. Assign, for each two-terminal graph $G$ with terminals $s$ and $t$, a simple graph $\hat{G}$ given by $G-s-t$. When $n\geq 4$ and $2n-3< m\leq \binom{n}{2}$, it is known that a two-terminal graph in $T_{n,m}$ is in $T_{n,m}^{(3)}$ if and only if its terminals are universal vertices~\cite{Bertrand-2018}. Xie et al.~\cite{Xie-2021} further proved that 
a two-terminal graph $G$ is in $T_{n,m}^{(4)}$ if and only if both its terminals are universal vertices in $G$ and $\hat{G}$ has the maximum number of $3$-paths among all simple graphs in  $\mathcal{G}_{n-2,m-2n+3}$. Observe that the equality $M_1(\hat{G})=2p_3(\hat{G})+2m$ holds for all graphs $\hat{G}$ in $\mathcal{G}_{n,m}$. Therefore, a two-terminal graph $G$ is in $T_{n,m}^{(4)}$ if and only if its terminals are universal vertices and $\hat{G}$ maximizes its first Zagreb index among all simple graphs in  $\mathcal{G}_{n-2,m-2n+3}$. Gong and Lin~\cite{Gong-2024} further proved that a two-terminal graph $G$ in $T_{n,m}^{(4)}$ is in $T_{n,m}^{(5)}$ if and only if $\hat{G}$ maximizes the invariant $H(\hat{G})$ given by $M_2(\hat{G})-6k_3(\hat{G})$, where $M_2(\hat{G})$ is the second Zagreb index of $\hat{G}$ and $k_3(\hat{G})$ is the number of triangles in $\hat{G}$. The following definitions given by Gong and Lin will be useful for our purposes.
\begin{definition}[Gong and Lin~\cite{Gong-2024}] 
A graph $G$ in $\mathcal{G}_{n,m}$ is $M$-optimal if $M_1(G)\geq M_1(G')$ for each $G'$ in $\mathcal{G}_{n,m}$.    
\end{definition}

\begin{definition}[Gong and Lin~\cite{Gong-2024}]\label{def:Hoptimal}
An $M$-optimal graph $G$ in $\mathcal{G}_{n,m}$ is $H$-optimal if 
$H(G)\geq H(G')$ for each $M$-optimal graph $G'$ in $\mathcal{G}_{n,m}$.  
\end{definition}

Gong and Lin proved the following lemma.
\begin{lemma}[Gong and Lin~\cite{Gong-2024}]\label{lemma:Hoptimality}
Let $n$ and $m$ be integers such that $n\geq 4$ and $2n-3<m\leq \binom{n}{2}$. A two-terminal graph $G$ in $T_{n,m}$ is in $T_{n,m}^{(5)}$ if and only if its terminal vertices $s$ and $t$ are universal and the simple graph $\hat{G}$ given by $G-s-t$ is $H$-optimal.    
\end{lemma}

In this article it is proved that in each nonempty set $\mathcal{G}_{n,m}$ there exists a unique $H$-optimal graph, that will be called $H_{n,m}$. We will also construct  $H_{n,m}$ (see Theorem~\ref{theorem:unique}). Then, we will use $H_{n,m}$ to characterize LMRTTG (see Theorem~\ref{theorem:main}). In fact, let $n$ and $m$ be integers such that $n\geq 4$ and $5\leq m\leq \binom{n}{2}$. If $5\leq m \leq 2n-3$ then Lemma~\ref{lemma:onlyH} gives the only LMRTTG in $T_{n,m}$. Otherwise, define $G_{n,m}$ as the two-terminal graph such that its terminals are universal and $\hat{G}_{n,m}$ equals the simple graph $H_{n-2,m-2n+3}$. As $H_{n,m}$ is the only $H$-optimal graph in $\mathcal{G}_{n-2,m-2n+3}$, Lemma~\ref{lemma:Hoptimality} gives that $T_{n,m}^{(5)}=\{G_{n,m}\}$. Finally, as $T_{n,m}^{(5)}=\{G_{n,m}\}$, Remark~\ref{remark:singleton} gives that $G_{n,m}$ is the only LMRTTG in $T_{n,m}$ thus proving Theorem~\ref{theorem:main}.\\   

We remark that $H$-optimal graphs are necessarily $M$-optimal (see Definition~\ref{def:Hoptimal}). As a consequence, we first revisit the body of related work dealing with $M$-optimality. Let $n$ be any positive integer and $m$ be any integer such that and $0\leq m\leq \binom{n}{2}$. Define the only integers $j$, $k$, $j'$ and $k'$ such that $1\leq j \leq k$,  $1\leq j' \leq k'$, and 
\begin{align}
m &= \binom{k+1}{2}-j; \label{eq:kj} \\
m &= \binom{n}{2}-\binom{k'+1}{2}+j'. \label{eq:kj'}
\end{align}
Let us define the following graphs, 
\begin{align*}
C_{n,m}^{1} &= (K_{k-j} \vee (K_1 \cup K_j)) \cup (n-k-1)K_1,\\
C_{n,m}^{2} &= ( K_1 \vee (K_{k-1} \cup (k-j)K_1) ) \cup (n-2k+j)K_1, \text{ if } k+1\leq 2k-j-1 \leq n-1,\\
C_{n,m}^{3} &= (K_{k-2}\vee 3K_1) \cup (n-k-1)K_1, \text{ if } j=3,\\
S_{n,m}^{1} &= K_{n-k'-1} \vee ( (K_1 \vee j'K_1) \cup (k'-j')K_1),\\
S_{n,m}^{2} &= K_{n-2k'+j'} \vee ( (K_{k'-j'} \vee (k'-1)K_1) \cup K_1), \text{ if } k'+1\leq 2k'-j'-1\leq n-1,\\
S_{n,m}^{3} &= K_{n-k'-1} \vee ( K_3 \cup (k'-2)K_1), \text{ if } j'=3.
\end{align*}
Let us define the set $\mathcal{O}_{n,m}$ consisting of each of the previous graphs in $\mathcal{G}_{n,m}$, whenever exist. The graphs $C_{n,m}^1$ and $S_{n,m}^1$ are known as the \emph{quasi-complete} and \emph{quasi-star}, respectively. For each $i\in \{1,2,3\}$ it holds that 
$\overline{S^i_{n,m}} = C_{n,\binom{n}{2}-m}^{i}$. Peled et al.~\cite{1999-Peled} and  Byer~\cite{Byer-1999} independently proved that each $M$-optimal graph in $\mathcal{G}_{n,m}$ belongs to the set $\mathcal{O}_{n,m}$. Byer also proved the following result using counting techniques.

\begin{lemma}[Theorem 4.3 and Corollary 4.4 in~\cite{Byer-1999}]\label{lemma:ByerMoptimality}
Let $n$ be any positive integer and let $m$ be any integer such that $0\leq m \leq n-1$. 
Let $m'=\binom{n}{2}-m$. The following assertions hold.
\begin{enumerate}[label=(\roman*)]
\item If $m\neq 3$ then $S_{n,m}^1$ (resp. $C_{n,m'}^1$) is the only $M$-optimal graph in $\mathcal{G}_{n,m}$ (resp. in $\mathcal{G}_{n,m'}$).
\item If $m=3$ then $S_{n,m}^1$ and $C_{n,m}^1$ (resp. $S_{n,m'}^1$ and $C_{n,m'}^1$) are the only $M$-optimal graphs in $\mathcal{G}_{n,m}$ (resp. in $\mathcal{G}_{n,m'}$).
\end{enumerate}
\end{lemma}

A characterization of $M$-optimal graphs was given later by \'Abrego et al.~\cite{Abrego-2009}. Their main results concerning $M$-optimality are given in Lemma~\ref{lemma:C1S1} and Lemma~\ref{lemma:Moptimal}.

\begin{lemma}[Theorem 2.4 in~\cite{Abrego-2009}]\label{lemma:C1S1}
 Let $n$ be any positive integer and let $m$ be any integer such that $0\leq m \leq \binom{n}{2}$. 
 All the following assertions simultaneously hold. 
\begin{enumerate}[label=(\roman*)]
\item Each $M$-optimal graph in $\mathcal{G}_{n,m}$ belongs to $\mathcal{O}_{n,m}$.
\item At least one of $C_{n,m}^1$ or $S_{n,m}^{1}$ is $M$-optimal.  
\item If $C_{n,m}^1$ is $M$-optimal then $C_{n,m}^2$ and $C_{n,m}^3$ are $M$-optimal, whenever these graphs exist.
\item If $S_{n,m}^1$ is $M$-optimal then $S_{n,m}^2$ and $S_{n,m}^3$ are $M$-optimal, whenever these graphs exist.
\end{enumerate}
\end{lemma}

For each integer $n$ such that $n\geq 5$ we define the  only integers $k_n$ and $\alpha_n$, and the only real numbers $q_n$ and $R_n$ satisfying the following relations,
\begin{align}
\binom{k_n}{2} &\leq \frac{1}{2}\binom{n}{2} < \binom{k_n+1}{2}, \label{eqkn}\\
\alpha_n       &= \binom{k_n}{2}, \label{eqalphan}\\
q_n            &= \frac{1-2(2k_n-3)^2+(2n-5)^2}{4}, \label{eqqn}\\
R_n            &= \frac{4(\binom{n}{2}-2\binom{k_n}{2})(k_n-2)}{-1-2(2k_n-4)^2+(2n-5)^2}. \label{eqRn}
\end{align}

\begin{lemma}[Theorem 2.8 in~\cite{Abrego-2009}]\label{lemma:Moptimal}
Let $n$ and $m$ be integers such that $n\geq 5$ and $0\leq m \leq \binom{n}{2}$. Define the numbers $k_n$, $\alpha_n$, $q_n$ and $R_n$ by Equations~\eqref{eqkn}, \eqref{eqalphan}, \eqref{eqqn}, and \eqref{eqRn}, respectively. All the following assertions hold.
\begin{enumerate}[label=(\roman*)]
\item If $q_n>0$ then 
\begin{align*}
M_1(S_{n,m}^1) &\geq M_1(C_{n,m}^1) \text{ when } 0\leq m \leq \frac{1}{2}\binom{n}{2},\\
M_1(C_{n,m}^1) &\geq M_1(S_{n,m}^1) \text{ when } \frac{1}{2}\binom{n}{2}\leq m \leq \binom{n}{2},
\end{align*}
and $M_1(S_{n,m}^1)=M_1(C_{n,m}^1)$ if and only if $m \in \{0,1,2,3,\binom{n}{2},\binom{n}{2}-1,\binom{n}{2}-2,\binom{n}{2}-3,\frac{1}{2}\binom{n}{2}\}$ or $m \in \{\alpha_n,\binom{n}{2}-\alpha_n\}$ and 
$(2n-3)^2-2(2k_n-3)^2$ equals either $-1$ or $7$.
\item If $q_n=0$ then
\begin{align*}
M_1(S_{n,m}^1) &\geq M_1(C_{n,m}^1) \text{ when } 0\leq m \leq \frac{1}{2}\binom{n}{2},\\
M_1(C_{n,m}^1) &\geq M_1(S_{n,m}^1) \text{ when } \frac{1}{2}\binom{n}{2}\leq m \leq \binom{n}{2},
\end{align*}
and $M_1(S_{n,m}^1)=M_1(C_{n,m}^1)$ if and only if $m \in \{0,1,2,3,\binom{n}{2},\binom{n}{2}-1,\binom{n}{2}-2,\binom{n}{2}-3,\frac{1}{2}\binom{n}{2}\}$ $\cup$ $\{\alpha_n,\alpha_n+1,\ldots,\binom{n}{2}-\alpha_n\}$.
\item If $q_n<0$ then
\begin{align*}
M_1(S_{n,m}^1) &\geq M_1(C_{n,m}^1) \text{ when } 0\leq m \leq \frac{1}{2}\binom{n}{2}-R_n \text{ or } \frac{1}{2}\binom{n}{2}\leq m \leq \frac{1}{2}\binom{n}{2}+R_n,\\
M_1(C_{n,m}^1) &\geq M_1(S_{n,m}^1) \text{ when } \frac{1}{2}\binom{n}{2}-R_n\leq m \leq \frac{1}{2}\binom{n}{2} \text{ or } \frac{1}{2}\binom{n}{2}+R_n\leq m \leq \binom{n}{2},
\end{align*}
and $M_1(S_{n,m}^1)=M_1(C_{n,m}^1)$ if and only if $m\in \{0,1,2,3,\binom{n}{2},\binom{n}{2}-1,\binom{n}{2}-2,\binom{n}{2}-3\}\cup \{\frac{1}{2}\binom{n}{2},\frac{1}{2}\binom{n}{2}-R_n,\frac{1}{2}\binom{n}{2}+R_n\}$.
\end{enumerate}
\end{lemma}

The following corollary of Lemma~\ref{lemma:Moptimal} will be useful for our purpose.
\begin{lemma}[Corollary 2.10 in~\cite{Abrego-2009}]\label{lemma:different}
Let $n$ be an integer such that $n\geq 6$. The following assertions hold.
\begin{enumerate}[label=(\roman*)]
\item If $m$ is an integer such that $4\leq m < \frac{1}{2}\binom{n}{2}-\frac{n}{2}$ then $M_1(S_{n,m}^1)>M_1(C_{n,m}^1)$.
\item If $m$ is an integer such that $\frac{1}{2}\binom{n}{2}+\frac{n}{2}< m \leq \binom{n}{2}-4$ then $M_1(C_{n,m}^1)>M_1(S_{n,m}^1)$.
\end{enumerate}
\end{lemma}

In the remaining of this section we include a list of results concerning $H$-optimal graphs. 
Using Ferrers diagrams from threshold graphs, Gong and Lin~\cite{Gong-2024} proved Lemma~\ref{lemma:chains}.
\begin{lemma}[Corollaries 3.5 and 3.6 in~\cite{Gong-2024}]\label{lemma:chains}
Let $n$ and $m$ be integers such that $n\geq 4$ and $0\leq m \leq \binom{n}{2}$, where $m\neq 5$. The following assertions hold.
\begin{enumerate}[label=(\roman*)]
\item\label{chains1} The inequalities $H(S_{n,m}^2)>H(S_{n,m}^1)>H(S_{n,m}^3)$ hold whenever $S_{n,m}^2$ and $S_{n,m}^3$ exist.
\item\label{chains2} The inequalities $H(C_{n,m}^3)>H(C_{n,m}^1)>H(C_{n,m}^2)$ hold whenever $C_{n,m}^2$ and $C_{n,m}^3$ exist.
\end{enumerate}
\end{lemma}

During the proof of Lemma~\ref{lemma:chains}, the authors obtained the following identities.
\begin{align}
H(C_{n,m}^2) &= H(C_{n,m}^1)-(k-\frac{7}{2})(k-j)(k-j-1), \label{eqC2}\\
H(C_{n,m}^3) &= H(C_{n,m}^1)+3, \label{eqC3}\\
H(S_{n,m}^2) &= H(S_{n,m}^1)+(k'-\frac{7}{2})(k'-j')(k'-j'-1), \label{eqS2}\\
H(S_{n,m}^3) &= H(S_{n,m}^1)-3, \label{eqS3}
\end{align}
where $j$, $k$, $j'$ and $k'$ are the only integers satisfying Equations~\eqref{eq:kj} and \eqref{eq:kj'} such that $1\leq j\leq k$ and $1\leq j'\leq k'$. Note that if $m=5$ then $k=3$, $j=1$, and Equation~\eqref{eqC2} gives that $H(C_{n,m}^2)=H(C_{n,m}^1)+1>H(C_{n,m}^1)$. 

Gong and Lin proved that there exists a unique $H$-optimal graph  
in each of the infinitely many pairs of integers $(n,m)$ such that $n\geq 5$, 
$0\leq m \leq \binom{n}{2}$, and $M_1(S_{n,m}^1)\neq M_1(C_{n,m}^1)$. In fact, let us define the following sets of pairs of integers,
\begin{align*}
I   &= \{(n,m): n,m\in \mathbb{Z}, n\geq 5, 0\leq m \leq \binom{n}{2}\},\\
I^+ &= \{(n,m): (n,m) \in I, \, \, M_1(S_{n,m}^1)>M_1(C_{n,m}^1)\},\\
I^- &= \{(n,m): (n,m) \in I, \, \, M_1(C_{n,m}^1)>M_1(S_{n,m}^1)\},\\
I^* &= \{(n,m): (n,m) \in I, \, \, M_1(S_{n,m}^1)=M_1(C_{n,m}^1)\}.
\end{align*}

A consequence of Lemma~\ref{lemma:chains} is Lemma~\ref{lemma:Gong-H}, which is a  restatement of Theorem 3.8 in~\cite{Gong-2024}. 
\begin{lemma}[Theorem 3.8 in~\cite{Gong-2024}]\label{lemma:Gong-H}
Let $(n,m)$ be in $I^+ \cup I^-$. The following assertions hold.
\begin{enumerate}[label=(\roman*)]
\item If $(n,m)\in I^+$ and $S_{n,m}^{2}$ does (does not) exist, then $S_{n,m}^2$ (resp.  $S_{n,m}^1$) is the only $H$-optimal graph in $\mathcal{G}_{n,m}$.
\item If $(n,m)\in I^-$ and $C_{n,m}^{3}$ does (does not) exist, then $C_{n,m}^3$ (resp. $C_{n,m}^1$) is the only $H$-optimal graph in $\mathcal{G}_{n,m}$.
\end{enumerate}
\end{lemma}

Gong and Lin~\cite{Gong-2024} also proved that $M_2(G)=3k_3(G)+p_4(G)+2p_3(G)+m$. As $H(G)$ equals $M_2(G)-6k_3(G)$, we get that
\begin{equation}\label{eq:H1}
H(G) = -3k_3(G)+p_4(G)+2p_3(G)+m.    
\end{equation}
The following Ramsey-type identities obtained by Byer~\cite{Byer-1999} hold for each pair $(n,m)$ in $I$. 
These identities will be useful to find $H$-optimal graphs.
\begin{align}
k_3(G)+k_3(\overline{G}) &= \binom{n}{3}-m(n-2) + p_3(G), \label{eq:k3}\\
p_3(G)+p_3(\overline{G})&=2p_3(G)+(n-2)(\binom{n}{2}-2m), \label{eq:p3}\\
p_4(G)+p_4(\overline{G}) &= 2(n-5)p_3(G)+2m^2-8m+3mn-3\binom{n}{3}+(n-2)^2(\binom{n}{2}-3m). \label{eq:p4}    
\end{align}

\section{Main result}\label{section:main}
In this section we will prove that each nonempty set $\mathcal{G}_{n,m}$ has a unique  $H$-optimal graph. Additionally we will find, in each nonempty set  $\mathcal{G}_{n,m}$, the only $H$-optimal graph that will be called $H_{n,m}$. As a consequence we will construct, in each set $T_{n,m}$ such that $n\geq 4$ and $5\leq m \leq \binom{n}{2}$, the only two-terminal graph $G_{n,m}$ in $T_{n,m}$ that is LMRTTG. For ease of readability, the proofs of each of the lemmas included in this section are given in the Appendix. A method to find $H(G)$ for each graph $G$ in $\mathcal{O}_{n,m}$ as well as  the proof strategy of the main result of this section are given in the following paragraphs.\\

Let $n$ and $m$ be any pair of integers such that $(n,m)$ is in $I$. If we can find closed forms for both $H(C_{n,m}^1)$ and $H(S_{n,m}^1)$ then, using Equations~\eqref{eqC2}, \eqref{eqC3}, \eqref{eqS2}, and \eqref{eqS3}, it is possible to find $H(G)$ for each graph $G$ in $\mathcal{O}_{n,m}$. Lemma~\ref{lemma:C1} gives a closed form for $H(C_{n,m}^1)$, while Lemma~\ref{lemma:ZagregSnm1} gives $M_1(S_{n,m}^1)$. Once we find $H(C_{n,m}^1)$ and $M_1(S_{n,m}^1)$, we can obtain $H(S_{n,m}^1)$ using the fact that $\overline{S_{n,m}^1}=C_{n,\binom{n}{2}-m}$ followed by an application of Lemma~\ref{lemma:Hsum}.\\

Now, let us present the proof strategy of the main result of this section. 
By Lemma~\ref{lemma:Gong-H}, the study of $H$-optimality can be reduced, among all pairs in $I$, to those pairs in $I^*$. The study of all pairs $(n,m)$ in $I^*$ such that $m\in \{0,1,2,3,\binom{n}{2},\binom{n}{2}-1,\binom{n}{2}-2,\binom{n}{2}-3\}$ is accomplished in Lemma~\ref{lemma:Byer}. The few pairs $(n,m)$ in $I^*$ that were not considered in Lemma~\ref{lemma:Byer} such that $n\in \{5,6,7\}$ are treated separately in Lemma~\ref{lemma:7}. Finally, observe that by Lemma~\ref{lemma:different}, each of the pairs $(n,m)$ in $I^*$ not previously studied belongs to the set $J$ defined as follows,
\begin{equation*}
J = \left\{(n,m): n,m\in \mathbb{Z}, \, n\geq 8, \, \frac{1}{2}\binom{n}{2}-\frac{n}{2}\leq  m \leq \frac{1}{2}\binom{n}{2}+\frac{n}{2}\right\}.    
\end{equation*}
The analysis of $H$-optimality for pairs $(n,m)$ in $I^*\cap J$ is more involved. 
First we will prove that, when $n$ is large enough,  
$C_{n,m}^3$ is the only $H$-optimal graph in $\mathcal{G}_{n,m}$ whenever exists or $C_{n,m}^1$ is the only $H$-optimal graph in $\mathcal{G}_{n,m}$ otherwise. 
The only $H$-optimal graph in $\mathcal{G}_{n,m}$ for the remaining finitely many pairs $(n,m)$ in $I^*\cap J$ is obtained by exhaustive computation.\\

In order to find $H(C_{n,m}^1)$ we can calculate the second Zagreb index of the quasi-star graph $C_{n,m}^1$ and subtract six times its number of triangles. The result is presented in Lemma~\ref{lemma:C1}. 
\begin{lemma}\label{lemma:C1}
Let $n$ and $m$ be integers such that $n\geq 1$ and $0\leq m \leq \binom{n}{2}$. Define $j$ and $k$ as the only integers such that $1\leq j\leq k$ satisfying  Equation~\eqref{eq:kj}. Then,
\begin{equation*}
H(C_{n,m}^1)= \frac{1}{2}k^4-\frac{1}{2}k^3-3jk^2+(j^2+7j+1)k-\frac{5j^2+7j}{2}.
\end{equation*}
\end{lemma}

The determination of the first Zagreb index of $S_{n,m}^1$ is straightforward from its definition.
\begin{lemma}\label{lemma:ZagregSnm1}
Let $n$ and $m$ be integers such that $n\geq 1$ and $0\leq m \leq \binom{n}{2}$. Define the only integers $j'$ and $k'$ satisfying Equation~\eqref{eq:kj'} such that $1\leq j'\leq k'$. Then,
\begin{equation*}
M_1(S_{n,m}^1) = (n-k'-1+j')^2+j'(n-k')^2+(k'-j')(n-k'-1)^2+(n-k'-1)(n-1)^2.    
\end{equation*}
\end{lemma}

Lemma~\ref{lemma:Hsum} is in fact a correction of a closed form for $H(G)+H(\overline{G})$ that appeared in~\cite{Gong-2024}. The previous correction does not alter any of the conclusions reached by Gong and Lin.  
\begin{lemma}\label{lemma:Hsum}
For each pair $(n,m)$ in $I$ and for each graph $G$ in $\mathcal{G}_{n,m}$,
\begin{equation}\label{eq:H}
H(G)+H(\overline{G}) = (n-\frac{9}{2})M_1(G)+h(n,m),
\end{equation}
where $h(n,m)=2m^2-6\binom{n}{3}+(n-1)^2\binom{n}{2}-3(n-1)(n-3)m$.
\end{lemma}

The following result is a consequence of Lemma~\ref{lemma:ByerMoptimality} and the Ramsey-type identity for $H(G)$ given in Lemma~\ref{lemma:Hsum}. 

\begin{lemma}\label{lemma:Byer}
For each pair $(n,m)$ in $I$ such that $m\in \{0,1,2,3,\binom{n}{2},\binom{n}{2}-1,\binom{n}{2}-2,\binom{n}{2}-3\}$ it holds that $S_{n,m}^1$ is the only 
$H$-optimal graph in $\mathcal{G}_{n,m}$. 
\end{lemma}

The proof of the following result is a simple exercise in combinatorics. 
\begin{lemma}\label{lemma:7}
There exists precisely $7$ pairs $(n,m)$ in $I^*$ such that $n\in \{5,6,7\}$ and further $m\notin \{i,\binom{n}{2}-i: i\in \{0,1,2,3\}\}$, namely, $(5,5)$, $(6,6),(6,7),(6,8),(6,9)$, $(7,9)$, and $(7,12)$. If $(n,m)$ is in $\{(5,5),(6,6),(6,8),(6,9)\}$ then $S_{n,m}^1$ is the only $H$-optimal graph in $\mathcal{G}_{n,m}$, and if 
$(n,m)$ is in $\{(6,7),(7,9),(7,12)\}$ then $S_{n,m}^2$ is the only $H$-optimal graph in $\mathcal{G}_{n,m}$.
\end{lemma}

From now on, we look for $H$-optimal graphs in $\mathcal{G}_{n,m}$ for those pairs $(n,m)$ in $I^* \cap J$. Lemma~\ref{lemma:kandkprime} is a technical result that will be useful to find the only $H$-optimal graph in each pair $(n,m)$ in $I^* \cap J$ when $n$ is sufficiently large (see Lemma~\ref{lemma:437}). The determination of the only $H$-optimal graph in the remaining finitely many pairs $(n,m)$ in $I^*$ will be accomplished by exhaustive computation. 

\begin{lemma}\label{lemma:kandkprime}
Let $(n,m)$ be in $J$. Define the only integers $j$, $k$, $j'$ and $k'$ such that 
$1\leq j\leq k$ and $1\leq j' \leq k'$ satisfying  Equations~\eqref{eq:kj} and \eqref{eq:kj'}. The following bounds for $k$ and $k'$ hold,
\begin{align*}
\frac{n}{\sqrt{2}}-2 &< k <\frac{n}{\sqrt{2}}+1,\\
\frac{n}{\sqrt{2}}-2 &< k' <\frac{n}{\sqrt{2}}+1.
\end{align*}
\end{lemma}

Define the polynomials $p(x)$ and $q(x)$ as follows, 
\begin{align}
p(x) &= \frac{1}{8}\left((3-2\sqrt{2})x^4-(20+37\sqrt{2})x^3 + (19+76\sqrt{2})x^2-(46+226\sqrt{2})x+72\right),\label{eq:p}\\
q(x) &= \frac{1}{8}\left( (2\sqrt{2})x^3-2x^2-16\sqrt{2}x+4\right). \label{eq:q}
\end{align}

\begin{lemma}\label{lemma:bounds}
Let $(n,m)$ be in $J$. All the following assertions hold.   
\begin{enumerate}[label=(\roman*)]
\item\label{i} $H(C_{n,m}^1)-H(S_{n,m}^1) \geq p(n)$.
\item\label{ii} For each pair of integers $i$ and $j$ in $\{1,2,3\}$ it holds that $|H(S_{n,m}^i)-H(S_{n,m}^j)|\leq q(n)$.
\end{enumerate}
\end{lemma}
Define the polynomial $r(x)$ as $p(x)-q(x)$, i.e., $r(x) = p(x)-q(x)$. 
\begin{remark}\label{remark:r}
By Sturm's theorem it can be easily proved that $r(x)$ has exactly one root in the interval $(436,437]$, which is the greatest real root of $r(x)$. As the leading coefficient of $r(x)$ is positive we conclude that $r(x)>0$ when $x\geq 437$.  
\end{remark}
Lemma~\ref{lemma:437} is a consequence of Lemma~\ref{lemma:bounds} and Remark~\ref{remark:r} (see the Appendix for details).
\begin{lemma}\label{lemma:437}
Let $(n,m)$ be in $I^* \cap J$ such that $n\geq 437$. The following assertions hold.
\begin{enumerate}[label=(\roman*)]
\item\label{part1} If $C_{n,m}^3$ exists then $C_{n,m}^3$ is the only $H$-optimal graph in $\mathcal{G}_{n,m}$.
\item\label{part2} If $C_{n,m}^3$ does not exist then $C_{n,m}^1$ is the only $H$-optimal graph in $\mathcal{G}_{n,m}$.
\end{enumerate}
\end{lemma}

Let $n$ and $m$ be fixed integers such that $(n,m)$ is in $I^*$. 
Recall that each $M$-optimal graph in $\mathcal{G}_{n,m}$ is in the set $\mathcal{O}_{n,m}$ which has at most $6$ graphs. 
By Lemma~\ref{lemma:C1} and Lemma~\ref{lemma:Hsum} we can find both 
$H(C_{n,m}^1)$ and $H(S_{n,m}^1)$. Furthermore, 
by means of Equations~\eqref{eqC2}, \eqref{eqC3}, \eqref{eqS2}, and \eqref{eqS3}, 
we can find $H(G)$ for each of the graphs in $\mathcal{O}_{n,m}$ and determine which 
graphs are $H$-optimal in $\mathcal{G}_{n,m}$. 
\begin{remark}
Computations among the finitely many pairs $(n,m)$ in $I^*\cap J$ such that $8\leq n \leq 436$ show that the same conclusions of Lemma~\ref{lemma:437} hold for all pairs $(n,m)$ in $I^* \cap J$. 
\end{remark}
 
Even though the determination of $H$-optimal graphs was carried out for all classes $\mathcal{G}_{n,m}$ such that $(n,m)$ is in $I$, the determination of $H$-optimal graphs in each of the nonempty classes $\mathcal{G}_{n,m}$ such that $(n,m)$ is not in $I$ is  elementary.

\begin{remark}\label{remark:trivial}
If $\mathcal{G}_{n,m}$ is nonempty but $(n,m)$ is not in $I$ then $S_{n,m}^1$ is the only $H$-optimal graph in $\mathcal{G}_{n,m}$. In fact, if $n\in \{1,2,3\}$ then $\mathcal{G}_{n,m}=\{S_{n,m}^1\}$ and the result trivially holds. Otherwise, $n=4$, and Lemma~\ref{lemma:ByerMoptimality} gives that $S_{n,m}^1$ is the only $H$-optimal graph in $\mathcal{G}_{n,m}$.
\end{remark}

We are in position to characterize the set consisting of all $H$-optimal graphs. 

\begin{theorem}\label{theorem:unique}
In each nonempty set $\mathcal{G}_{n,m}$ there exists a unique $H$-optimal graph, that we denote $H_{n,m}$. Additionally, all the following assertions hold.
\begin{enumerate}[label=(\roman*)]
\item If $(n,m)\notin I$ then $H_{n,m}=S_{n,m}^1$. 
\item If $(n,m)\in I^+$ and $S_{n,m}^2$ does (does not) exist then $H_{n,m}=S_{n,m}^2$ (resp. $H_{n,m}=S_{n,m}^1$).
\item If $(n,m)\in I^-$ and $C_{n,m}^3$ does (does not) exist then $H_{n,m}=C_{n,m}^3$ 
(resp. $H_{n,m}=C_{n,m}^1$).  
\item If $(n,m) \in I^*$ and $(n,m)$ is in $\{(5,5),(6,6),(6,8),(6,9)\}$ then $H_{n,m}=S_{n,m}^1$.
\item If $(n,m) \in I^*$ and $(n,m)$ is in $\{(6,7),(7,9),(7,12)\}$ then $H_{n,m}=S_{n,m}^2$.
\item If $(n,m)\in I^*$ and $m\in \{i,\binom{n}{2}-i: i\in \{0,1,2,3\}\}$ then $H_{n,m}=S_{n,m}^1$.
\item If $(n,m)\in I^* \cap J$ and $C_{n,m}^3$ does (does not) exist then $H_{n,m}=C_{n,m}^3$ (resp. $H_{n,m}=C_{n,m}^1$). 
\end{enumerate}
\end{theorem}

We are ready to prove the main result of this article.
\begin{theorem}\label{theorem:main}
In each set $T_{n,m}$ such that $n\geq 4$ and $5\leq m \leq \binom{n}{2}$ there exists a unique LMRTTG. 
Additionally, the following assertions hold.
\begin{enumerate}[label=(\roman*)]
\item\label{teo1} If $5\leq m\leq 2n-3$ then $G_{n,m}$ given in Definition~\ref{definition:Hnm} is the only $LMRTTG$ in $T_{n,m}$.
\item\label{teo2} If $2n-3 < m \leq \binom{n}{2}$ then the two-terminal graph $G_{n,m}$ in $T_{n,m}$ whose terminals are universal such that $G_{n,m}-s-t$ is $H_{n-2,m-(2n-3)}$ is the only LMRTTG in $T_{n,m}$. 
\end{enumerate}
\end{theorem}
\begin{proof}
It is enough to prove Assertions~\ref{teo1} and \ref{teo2}. On the one hand, Assertion~\ref{teo1} is Lemma~\ref{lemma:onlyH}. On the other hand, let $n$ and $m$ be integers such that 
$n\geq 4$ and $2n-3<m \leq \binom{n}{2}$. Let $n'=n-2$ and $m'=m-(2n-3)$. As $n\geq 4$ and $2n-3<m\leq \binom{n}{2}$, we get that $n'\geq 2$ and $0\leq m' \leq \binom{n}{2}-(2n-3) = \binom{n-2}{2} = \binom{n'}{2}$. Therefore, $\mathcal{G}_{n',m'}$ is a nonempty set of simple graphs. By Theorem~\ref{theorem:unique}, there exists a unique $H$-optimal graph in $\mathcal{G}_{n',m'}$, that we call $H_{n',m'}$. Define the only two-terminal graph $G_{n,m}$ in $T_{n,m}$ such that its terminals $s$ and $t$ are universal vertices in $G_{n,m}$ and further $G_{n,m}-s-t=H_{n',m'}$. As $H_{n',m'}$ is the only $H$-optimal graph 
in $\mathcal{G}_{n',m'}$, by Lemma~\ref{lemma:Hoptimality} we know that $T_{n,m}^{(5)}=\{G_{n,m}\}$. Applying 
Remark~\ref{remark:singleton} with $j=5$ we get that $G_{n,m}$ is the only LMRTTG in $T_{n,m}$ thus proving Assertion~\ref{teo2} as required.  
\end{proof}

\section*{Acknowledgments}
The author wants to thank Prof. Mart\'in Safe for his insightful comments that improved the presentation of this manuscript. The author also wants to thank Prof. Claudio Qureshi who kindly provided me with a proof that the denominator that appears on the right-hand of Equation~\eqref{eqRn} never gets null for any choice of integer $n$ such that $n\geq 5$.

\section{Appendix}
\emph{Proof of Lemma~\ref{lemma:C1}}: let $n$ and $m$ be integers such that $n\geq 1$ and $0\leq m\leq \binom{n}{2}$. Define the only integers $k$ and $j$ satisfying Equation~\eqref{eq:kj} such that $1\leq j\leq k$. The quasi-star graph $C_{n,m}^1$ arises  from $K_{k}$ by the addition of one vertex joined to precisely $k-j$ of the vertices in $K_k$, and $n-k-1$ isolated vertices. Let $v_1,v_2,\ldots,v_k$ be the $k$ vertices in $K_k$ and let $v_{k+1}$ be the vertex that is joined precisely to $v_{1},v_{2},\ldots,v_{k-j}$. 
Define $V_1=\{v_1,v_2,\ldots,v_{k-j}\}$ and $V_{2}=\{v_{k-j+1},v_{k-j+2},\ldots,v_{k}\}$. The degree of each vertex in $V_1$ or $V_2$ equals $k$ or $k-1$ respectively, while the the degree of 
$v_{k+1}$ equals $k-j$. 
The edge set of $C_{n,m}^1$ can be partitioned into precisely $4$ sets: the edge set $E_1$ whose endpoint are both in $V_1$; the edge set $E_2$ whose endpoints are both in $V_2$; the edge set $E_3$ such that one endpoint is in $V_1$ and one endpoint is in $V_2$; and the edge set $E_4$ such that one endpoint is in $V_1$ and one endpoint is $v_{k+1}$. Clearly, $E(C_{n,m}^1)=E_1\cup E_2 \cup E_3 \cup E_4$. The size of each edge set is $|E_1|=\binom{k-j}{2}$, $|E_2|=\binom{j}{2}$, $|E_3|=j(k-j)$, and  $|E_4|=k-j$. On the one hand, 
\begin{align*}
M_2(C_{n,m}^1) &= \sum_{vw\in E_1}d(v)d(w)+\sum_{vw\in E_2}d(v)d(w)+
\sum_{vw\in E_3}d(v)d(w)+\sum_{vw\in E_4}d(v)d(w)\\
&= \binom{k-j}{2} k^2 + \binom{j}{2}(k-1)^2 + j(k-j)k(k-1) + k(k-j)^2\\
&= \frac{1}{2}k^2(k-j)(k-j-1)+\frac{1}{2}j(j-1)(k-1)^2+j(k-j)k(k-1)+k(k-j)^2\\
&= \frac{1}{2}k^2(k^2-(2j+1)k+j(j+1))+\frac{1}{2}j(j-1)(k^2-2k+1)\\
&+jk(k^2-(j+1)k+j)+k(k^2-2jk+j^2)\\
&= \frac{1}{2}k^4 +\frac{1}{2}k^3 -3jk^2 + (j+j^2)k + \frac{1}{2}j(j-1). 
\end{align*}
On the other hand, each triangle in $C_{n,m}^1$ has either its $3$ vertices in $V_1 \cup V_2$ or one of its vertices is $v_{k+1}$ and the other $2$ vertices are in $V_1$ thus 
\begin{align*}
k_3(C_{n,m}^1) &= \binom{k}{3}+\binom{k-j}{2} = \frac{1}{6}(k(k-1)(k-2)+3(k-j)(k-j-1))\\
&= \frac{1}{6}(k^3-3k^2+2k+3k^2-3(2j+1)k+3j(j+1))= \frac{1}{6}(k^3-(6j+1)k+ 3j(j+1)).
\end{align*}
Therefore, $H(C_{n,m}^1)=M_2(C_{n,m}^1)-6k_3(C_{n,m}^1)=\frac{1}{2}k^4-\frac{1}{2}k^3-3jk^2 + (j^2+7j+1)k-\frac{5j^2+7j}{2}$. \qed\\

\emph{Proof of Lemma~\ref{lemma:ZagregSnm1}}: let $n$ and $m$ be integers such that $n\geq 1$ and $0\leq m \leq \binom{n}{2}$. Define the only integers $k'$ and $j'$ satisfying  Equation~\eqref{eq:kj'} such that $1\leq j'\leq k'$. Recall that $S_{n,m}^1=K_{n-k'-1} \vee ( (K_1 \vee j'K_1) \cup (k'-j')K_1)$. Then, $S_{n,m}^1$ has precisely $1$ vertex with degree $n-k'-1+j'$; $j'$ vertices with degree $n-k'$; $k'-j'$ vertices with degree $n-k'-1$; and $n-k'-1$ vertices with degree $n-1$. By the definition of the first Zagreb index,
\begin{equation*}
M_1(S_{n,m}^1) = (n-k'-1+j')^2+j'(n-k')^2+(k'-j')(n-k'-1)^2+(n-k'-1)(n-1)^2. \qed    
\end{equation*}

\emph{Proof of Lemma~\ref{lemma:Hsum}}: let $(n,m)$ be in $I$ and let $G$ be any graph in $\mathcal{G}_{n,m}$. Applying Equation~\eqref{eq:H1} to both graphs $G$ and $\overline{G}$ yields
\begin{align*}
H(G) &= -3k_3(G)+p_4(G)+2p_3(G)+m;\\
H(\overline{G}) &= -3k_3(\overline{G})+p_4(\overline{G})+2p_3(\overline{G})+(\binom{n}{2}-m).
\end{align*}
Now, summing the previous expressions and using Equations~\eqref{eq:k3}, \eqref{eq:p3}, and \eqref{eq:p4} yields
\begin{align*}
H(G)+H(\overline{G}) &= -3(k_3(G)+k_3(\overline{G}))+(p_4(G)+p_4(\overline{G}))+2(p_3(G)+p_3(\overline{G})) + \binom{n}{2}\\
&= -3(\binom{n}{3}-m(n-2) + p_3(G))\\
&+ 2(n-5)p_3(G)+2m^2-8m+3mn-3\binom{n}{3}+(n-2)^2(\binom{n}{2}-3m)\\
&+ 4p_3(G)+2(n-2)(\binom{n}{2}-2m)) + \binom{n}{2}\\
&= (2n-9)p_3(G)+2m^2-6\binom{n}{3}+(n-1)^2\binom{n}{2}-(3n^2-14n+18)m,
\end{align*}
Now, observe that $p_3(G)$ counts the number of pairs of adjacent edges in $G$. Therefore, $p_3(G)=\sum_{i=1}^{n}\binom{d_i(G)}{2}$ and $2p_3(G)=\sum_{i=1}^{n}d_i(G)(d_i(G)-1) = M_1(G)-2m$. Consequently, $(2n-9)p_3(G)=(n-\frac{9}{2})(M_1(G)-2m)=(n-\frac{9}{2})M_1(G)-2mn+9m$, and 
\begin{align*}
H(G)+H(\overline{G}) &= (2n-9)p_3(G)+2m^2-6\binom{n}{3}+(n-1)^2\binom{n}{2}-(3n^2-14n+18)m\\
&= (n-\frac{9}{2})M_1(G)+2m^2-6\binom{n}{3}+(n-1)^2\binom{n}{2}-(3n^2-12n+9)m\\
&= (n-\frac{9}{2})M_1(G) + 2m^2-6\binom{n}{3}+(n-1)^2\binom{n}{2}-3(n-1)(n-3)m\\  
&= (n-\frac{9}{2})M_1(G) + h(n,m),
\end{align*}
where $h(n,m)=2m^2-6\binom{n}{3}+(n-1)^2\binom{n}{2}-3(n-1)(n-3)m$, as required.
\qed\\

\emph{Proof of Lemma~\ref{lemma:Byer}}: let $(n,m)$ be in $I$. If $m\in \{0,1,2,\binom{n}{2},\binom{n}{2}-1,\binom{n}{2}-2\}$ then Lemma~\ref{lemma:ByerMoptimality} gives that $S_{n,m}^1$ is the only $M$-optimal graph in $\mathcal{G}_{n,m}$ thus it is the only $H$-optimal graph in $\mathcal{G}_{n,m}$. If $m=3$ then Lemma~\ref{lemma:ByerMoptimality} gives that the only $M$-optimal graphs in $\mathcal{G}_{n,m}$ are $S_{n,m}^1$ or $C_{n,m}^1$. As $m=3$, the quasi-star graph $S_{n,m}^1$ equals the star $K_{1,3}$ plus $n-4$ isolated vertices, while the quasi-complete graphs $C_{n,m}^1$ equals the complete graph $C_3$ plus $n-4$ isolated vertices. On the one hand, as $S_{n,m}^1$ has precisely $3$ edges each one whose endpoint vertices have degrees $1$ and $3$, we get that $M_2(S_{n,m}^1)=3+3+3=9$. As $S_{n,m}^1$ has no triangles, $H(S_{n,m}^1)=M_2(S_{n,m}^1)=9$. On the other hand, as $C_{n,m}^1$ has precisely $3$ edges each one whose endpoint vertices have degrees $2$, we get that $M_2(C_{n,m}^1)=3\times 2^2=12$. As $C_{n,m}^1$ has precisely $1$ triangle, $H(C_{n,m}^1)=12-6=6$. Consequently, $S_{n,m}^1$ is the only $H$-optimal graph in $\mathcal{G}_{n,m}$. Finally, if $m=\binom{n}{2}-3$ then 
Lemma~\ref{lemma:ByerMoptimality} gives that  $C_{n,m}^1$ and $S_{n,m}^1$ are the only $M$-optimal graphs in $\mathcal{G}_{n,m}$. Then, $M_1(C_{n,m}^1)=M_1(S_{n,m}^1)$, and by Lemma~\ref{lemma:Hsum},  
$H(S_{n,\binom{n}{2}-3}^1)+H(\overline{S_{n,\binom{n}{2}-3}^1})=H(C_{n,\binom{n}{2}-3}^1)+H(\overline{C_{n,\binom{n}{2}-3}^1})$. Using that 
$\overline{S_{n,\binom{n}{2}-3}^1}=C_{n,3}^1$ and  
$\overline{C_{n,\binom{n}{2}-3}^1}=S_{n,3}^1$, we conclude that 
$H(S_{n,\binom{n}{2}-3}^1)-H(C_{n,\binom{n}{2}-3}^1)=H(S_{n,3}^1)-H(C_{n,3}^1)=3$ and 
$S_{n,m}^1$ is the only $H$-optimal graph in $\mathcal{G}_{n,m}^1$ as required. \qed \\

\emph{Proof of Lemma~\ref{lemma:7}}: first, let us prove that there exists precisely $7$ pairs $(n,m)$ in $I^*$ such that 
$n\in \{5,6,7\}$ and $m\notin \{i,\binom{n}{2}-i: i\in \{0,1,2,3\}\}$. 
If $n=5$ then Expression~\eqref{eqkn} yields $k_n=3$, while  
Equation~\eqref{eq:q} gives that $q_5=2$, which is positive. Observe that 
$(2n-3)^2-2(2k_n-3)^2=31$, which is not $-1$ neither $7$. Applying  Lemma~\ref{lemma:Moptimal} when $n=5$ gives that the only choice for $m$ meeting the conditions of the statement is $5$. If $n=6$ then $k_6=4$ while Equation~\eqref{eq:q} gives that $q_6=0$. By Equation~\eqref{eqalphan} we obtain that $\alpha_6=\binom{k_6}{2}=6$. Applying Lemma~\ref{lemma:Moptimal} when $n=6$ gives that the only choices 
for $m$ meeting the conditions of the statement are $6$,$7$, $8$, and $9$. Similarly, if $n=7$ then $k_7=5$ and $q_7=-4$, which is negative. 
Applying Lemma~\ref{lemma:Moptimal} when $n=7$ gives that the only choices for $m$ 
meeting the conditions of the statement is $9$ and $12$. Consequently, there are precisely $7$ pairs $(n,m)$ in $I^*$ meeting the conditions of the statement, namely, 
$(5,5)$, $(6,6),(6,7),(6,8),(6,9)$, $(7,9)$, and $(7,12)$, thus proving the first part of the lemma.

Now, let us find all $H$-optimal graphs in $\mathcal{G}_{n,m}$ in each of the $7$ pairs $(n,m)$ from the statement separately. 

\begin{enumerate}[label=(\roman*)]
\item If $(n,m)=(5,5)$ then Equation~\eqref{eq:kj} gives that 
$k=3$ and $j=1$. As $j\neq 3$ we know that $C_{n,m}^3$ does not exist. 
As $k+1\leq 2k-j-1\leq n-1$, $C_{n,m}^2$ exists. Observe that, as $m=5$, 
Equation~\eqref{eqC2} gives that 
$H(C_{5,5}^2)>H(C_{5,5}^1)$ and Equations~\eqref{eqS2} and \eqref{eqS3} give  that 
$H(S_{5,5}^1)>H(S_{5,5}^2)>H(S_{5,5}^3)$. By definition, the graph $C_{5,5}^2$ equals 
$K_1 \vee (K_2 \cup 2K_1)$, which is isomorphic to the quasi-star graph $S_{5,5}^1$  
thus $S_{5,5}^1$ is the only $H$-optimal in $\mathcal{G}_{5,5}$. 
\item If $(n,m)=(6,6)$ then $k=j=4$ while $k'=4$ and $j'=1$. Then, $C_{n,m}^3$ and $S_{n,m}^2$ do not exist. By Lemma~\ref{lemma:Gong-H}, it is enough to compare $H(S_{6,6}^1)$ with $H(C_{6,6}^1)$. On the one hand, the quasi-star graph $S_{6,6}^1$ has precisely three edges whose endpoints have degree $1$ and $5$, two edges whose endpoints have degrees $5$ and $2$, and one edge both whose endpoints have degree  $2$, thus $M_2(S_{6,6}^1)=3\times 1 \times 5+ 2\times 5\times 2+1\times 2\times 2=39$. As $S_{6,6}^1$ has only one triangle, $H(S_{6,6}^1)=39-6=33$. 
On the other hand, the quasi-complete graph $C_{6,6}^1$ equals $K_4$ plus two isolated vertices thus $C_{6,6}^1$ has only $4$ triangles and the endpoints of each of its six edges have degrees $3$ thus $H(C_{6,6}^1)=6\times 3^2-6\times 4 = 30$. As $H(S_{6,6}^1)>H(C_{6,6}^1)$, Lemma~\ref{lemma:Gong-H} gives that $S_{6,6}^1$ is the only $H$-optimal graph in $\mathcal{G}_{6,6}$. 
\item If $(n,m)=(6,7)$ then $k=4$ and $j=3$, while $k'=4$ and $j'=2$. Then, both graphs $C_{6,7}^3$ and $S_{6,7}^2$ exist. 
As both graphs $C_{6,7}^3$ and $S_{6,7}^2$ are isomorphic to $(K_2 \vee 3K_1)\cup K_1$, Lemma~\ref{lemma:Gong-H} gives that $S_{6,7}^2$ is the only $H$-optimal graph in $\mathcal{G}_{6,7}$.
\item If $(n,m)=(6,8)$ then $k=4$ and $j=2$, while $k'=4$ and $j'=3$. Then, $C_{6,8}^3$ and $S_{6,8}^2$ do not exist. On the one hand, $S_{6,8}^1$ has precisely one edge whose endpoints have degrees $4$ and $5$, one edge whose endpoints have degrees $1$ and $5$, three edges whose endpoints have degrees $2$ and $5$, and three edges whose endpoints have degrees $2$ and $4$, 
thus $M_2(S_{6,8}^1)=1\times 4\times 5 + 1\times 1\times 5 + 3\times 2\times 5 + 3\times 2\times 4 = 79$. As $S_{6,8}^1$ has precisely $3$ triangles, $H(S_{6,8}^1)=61$. On the other hand, $C_{6,8}^1$ 
has precisely one edge whose endpoints have degrees $4$ and $4$, one edge whose endpoints have degrees $3$ and $3$, two edges whose endpoints have degrees $2$ and $4$, and four edges whose endpoints have degrees $3$ and $4$ thus 
$M_2(C_{6,8}^1)=1\times 4\times 4 + 1\times 3\times 3 + 2\times 2\times 4 + 4\times 3\times 4 = 89$. As $C_{6,8}^1$ has precisely $5$ triangles, $H(C_{6,8}^1)=59$. 
As $H(S_{6,8}^1)>H(C_{6,8}^1)$, Lemma~\ref{lemma:Gong-H} gives that $S_{6,8}^1$ is the only $H$-optimal graph in $\mathcal{G}_{6,8}$. 
\item If $(n,m)=(6,9)$ then $k=4$ and $j=1$, while $k'=j'=4$. Then, $C_{6,9}^3$ and $S_{6,9}^2$ do not exist. Lemma~\ref{lemma:Hsum} gives that $H(S_{6,9}^1)+H(\overline{S_{6,9}^1})=H(C_{6,9}^1)+H(\overline{C_{6,9}^1})$. Using that $\overline{S_{6,9}^1}=C_{6,6}^1$ and $\overline{C_{6,9}^1}=S_{6,6}^1$ gives that $H(S_{6,9}^1)-H(C_{6,9}^1)=H(S_{6,6}^1)-H(C_{6,6}^1)=33-30=3$. 
As $H(S_{6,9}^1)>H(C_{6,9}^1)$, Lemma~\ref{lemma:Gong-H} gives that 
$S_{6,9}^1$ is the only $H$-optimal graph in $\mathcal{G}_{6,9}$. 
\item If $(n,m)=(7,9)$ then $k=4$ and $j=1$, while $k'=5$ and $j'=3$. Then, $C_{7,9}^3$ does not exist while $S_{7,9}^2$ exists. By Lemma~\ref{lemma:Gong-H}, it is enough to compare $H(S_{7,9}^2)$ with $H(C_{7,9}^1)$. On the one hand, $S_{7,9}^2$ equals $(K_2 \vee 4K_1) \cup K_1$ hence it has precisely $4$ triangles, the endpoints of eight of its edges have degrees $2$ and $5$, and the endpoint of one of its edges have degrees $5$ and $5$, thus $H(S_{7,9}^2)=8\times 5\times 2 + 1\times 5\times 5-6\times 4=81$. On the other hand, $C_{7,9}^1$ equals $K_5$ minus one edge plus two isolated vertices, hence it has $7$ triangles, the endpoints of six of its edges have degrees $4$ and $3$ while the endpoints of three of its edges have degrees $4$ and $4$, thus $H(C_{7,9}^1)=6\times 4\times 3 + 3\times 4\times 4-6\times 7=78$. 
Consequently, Lemma~\ref{lemma:Gong-H} gives that $S_{7,9}^2$ is the only $H$-optimal graph in $\mathcal{G}_{7,9}$.
\item If $(n,m)=(7,12)$ then $k=5$ and $j=3$, while $k'=4$ and $j'=1$. Then, both graphs $C_{7,12}^3$ and $S_{7,12}^2$ exist. As both graphs $C_{7,12}^3$ and $S_{7,12}^2$ are isomorphic to $(K_3 \vee 3K_1) \cup K_1$, Lemma~\ref{lemma:Gong-H} gives that $S_{7,12}^2$ is the only $H$-optimal graph in $\mathcal{G}_{7,12}$.
\end{enumerate}
By the previous study we conclude that if $(n,m)$ is in $\{(5,5),(6,6),(6,8),(6,9)\}$ then $S_{n,m}^1$ is the only $H$-optimal graph in $\mathcal{G}_{n,m}$, and if 
$(n,m)$ is in $\{(6,7),(7,9),(7,12)\}$ then $S_{n,m}^2$ is the only $H$-optimal graph in $\mathcal{G}_{n,m}$, as required. \qed\\

\emph{Proof of Lemma~\ref{lemma:kandkprime}}: 
as $(n,m)$ is in $J$ we know that $n\geq 8$ and $\frac{1}{2}\binom{n}{2}-\frac{n}{2}\leq m \leq \frac{1}{2}\binom{n}{2}+\frac{n}{2}$. Define the only integers $j$ and $k$ such that $1\leq j\leq k$ and $m=\binom{k+1}{2}-j$. On the one hand,
\begin{align*}
\binom{k+1}{2} &= \frac{1}{2}k(k+1) = m+j \geq m+1 \geq \frac{1}{2}\binom{n}{2}-\frac{n}{2}+1\\
&= \frac{1}{4}(n^2-3n+4) > \frac{1}{4}(n^2-\frac{6}{\sqrt{2}}n+4) = \frac{1}{2}\left(\frac{n}{\sqrt{2}}-2\right)\left(\frac{n}{\sqrt{2}}-1\right). 
\end{align*}
The function $f:\mathbb{R}^+ \to \mathbb{R}^+$ given by $f(x)=\frac{1}{2}x(x+1)$ is increasing. We already proved that $f(k)>f(\frac{n}{\sqrt{2}}-2)$ thus $k>\frac{n}{\sqrt{2}}-2$. 

On the other hand $\binom{k+1}{2}=m+j \leq \frac{1}{2}\binom{n}{2}+\frac{n}{2}+k$. 
Then,
\begin{align*}
\frac{1}{2}(k-1)k&= \binom{k+1}{2}-k \leq \frac{1}{2}\binom{n}{2}+\frac{n}{2} = \frac{1}{4}n(n+1) < \frac{1}{2}\frac{n}{\sqrt{2}}\left(\frac{n}{\sqrt{2}}+1\right),     
\end{align*}
and $f(k-1)<f(\frac{n}{\sqrt{2}})$ thus $k-1 < \frac{n}{\sqrt{2}}$, or equivalently, $k<\frac{n}{\sqrt{2}}+1$. Thus far, we proved that $\frac{n}{\sqrt{2}}-2<k<\frac{n}{\sqrt{2}}+1$. 

Finally, define the only integers $j'$ and $k'$ such that $1\leq j'\leq k'$ satisfying 
Equation~\eqref{eq:kj'}, that is, $m = \binom{n}{2}-\binom{k'+1}{2}+j'$. 
Let $m'=\binom{n}{2}-m$. Then $m'=\binom{k'+1}{2}-j'$. Observe that $m'$ satisfies that $\frac{1}{2}\binom{n}{2}-\frac{n}{2}\leq m' \leq \frac{1}{2}\binom{n}{2}+\frac{n}{2}$. The previous reasoning to prove that $\frac{n}{\sqrt{2}}-2<k<\frac{n}{\sqrt{2}}+1$ holds replacing $k$, $j$ and $m$ by $k'$, $j'$ and $m'$ respectively, and we get that $\frac{n}{\sqrt{2}}-2 < k' <\frac{n}{\sqrt{2}}+1$, as required. \qed\\

\emph{Proof of Lemma~\ref{lemma:bounds}}: as $(n,m)$ is in $J$ we know that $\frac{1}{2}\binom{n}{2}-\frac{n}{2}\leq m \leq \frac{1}{2}\binom{n}{2}+\frac{n}{2}$. Define the only integers $j$, $k$, $j'$ and $k'$ satisfying Equations~\eqref{eq:kj} and \eqref{eq:kj'} such that $1\leq j\leq k$ and $1\leq j' \leq k'$. 
Equations~\eqref{eq:kj} and \eqref{eq:kj'} give that $m=\binom{k+1}{2}-j$ and 
$m=\binom{n}{2}-\binom{k'+1}{2}+j'$. Let us prove each of the assertions in order. 

\begin{enumerate}[label=(\roman*)]
\item We want to prove that a lower bound for the function $H(C_{n,m}^1)-H(S_{n,m}^1)$ is the polynomial $p(n)$ given in Equation~\eqref{eq:p}. Our proof strategy is the following. First, we will find a lower bound $p_1(n)$ for $H(C_{n,m})$, i.e., 
$H(C_{n,m}^1)\geq p_1(n)$. Then, we will find upper bounds $p_2(n)$ and $p_3(n)$ for the 
functions $M_1(S_{n,m}^1)$ and $h(n,m)$, i.e., $M_1(S_{n,m}^1)\leq p_2(n)$ and $h(n,m)\leq p_3(n)$. Finally, define $p(n)=2p_1(n)-(n-\frac{9}{2})p_2(n)-p_3(n)$. As $n\geq 8$ we know in particular that $n-\frac{9}{2}\geq 0$, and 
Equation~\eqref{eq:H} gives that 
\begin{align}
H(C_{n,m}^1)-H(S_{n,m}^1)&= H(C_{n,m}^1)-((n-\frac{9}{2})M_1(S_{n,m}^1)+h(n,m)-H(C_{n,\binom{n}{2}-m}^1)) \nonumber\\
&=H(C_{n,m}^1)+H(C_{n,\binom{n}{2}-m}^1)-(n-\frac{9}{2})M_1(S_{n,m}^1)-h(n,m)\nonumber\\
&\geq 2p_1(n)-(n-\frac{9}{2})p_2(n)-p_3(n)=p(n), \label{eq:ineqH}
\end{align}
where we used that $\overline{S_{n,m}^1}=C_{n,\binom{n}{2}-m}^1$ and that the number $m'$ given by $\binom{n}{2}-m$ also satisfies that $(n,m')$ is in $J$ hence 
the inequality $H(C_{n,m'}^1)\geq p_1(n)$ holds.

First, let us find a polynomial $p_1(n)$ such that $H(C_{n,m}^1)\geq p_1(n)$. By Lemma~\ref{lemma:C1},
\begin{equation}\label{eq:C1exactly}
H(C_{n,m}^1)= \frac{1}{2}k^4-\frac{1}{2}k^3-3jk^2+(j^2+7j+1)k-\frac{5j^2+7j}{2}.
\end{equation}
Lemma~\ref{lemma:kandkprime} gives that $\frac{n}{\sqrt{2}}-2<k<\frac{n}{\sqrt{2}}+1$. We also know that 
$1\leq j \leq k$. Therefore, $-3jk^2 \geq -3k^3$, $(j^2+7j+1)k\geq 9k$, and $-\frac{5j^2+7j}{2}\geq -\frac{5k^2+7k}{2}$. Replacing the previous bounds into Equation~\eqref{eq:C1exactly} yields
\begin{align*}
H(C_{n,m}^1) &\geq \frac{1}{2}k^4-\frac{1}{2}k^3-3k^3+9k-\frac{5}{2}k^2-\frac{7}{2}k\\
&= \frac{1}{2}k^4-\frac{7}{2}k^3-\frac{5}{2}k^2+\frac{11}{2}k\\
&> \frac{1}{2}\left(\frac{n}{\sqrt{2}}-2 \right)^4 -  
\frac{7}{2}\left(\frac{n}{\sqrt{2}}+1 \right)^3 - 
\frac{5}{2}\left(\frac{n}{\sqrt{2}}+1 \right)^2+
\frac{11}{2}\left(\frac{n}{\sqrt{2}}-2 \right)\\
&= \frac{1}{2}\left(\frac{1}{4}n^4-2\sqrt{2}n^3+12n^2-16\sqrt{2}n + 16\right) - 
\frac{7}{2}\left(\frac{1}{2\sqrt{2}}n^3 + \frac{3}{2}n^2+\frac{3}{\sqrt{2}}n+1\right)\\
&- \frac{5}{2}\left(\frac{1}{2}n^2+\sqrt{2}n+1\right)+
\frac{11}{2}\left(\frac{n}{\sqrt{2}}-2 \right)\\
&= \frac{1}{8}(n^4-\frac{15}{8}\sqrt{2}n^3-\frac{1}{2}n^2-13\sqrt{2}n-9)=p_1(n).
\end{align*}

Recall that, by Lemma~\ref{lemma:ZagregSnm1},
\begin{equation}\label{eq:M1int}
M_1(S_{n,m}^1) = (n-k'-1+j')^2+j'(n-k')^2+(k'-j')(n-k'-1)^2+(n-k'-1)(n-1)^2.    
\end{equation}

As $k'<n$ we get that $(n-k'-1)^2<(n-k')^2$ and $j'(n-k')^2+(k'-j')(n-k'+1)^2\leq k'(n-k')^2$. As $j'\leq k'$ we get that $(n-k'-1+j')^2 \leq (n-1)^2$.
Replacing the previous inequalities into Equation~\eqref{eq:M1int} gives that
\begin{align}
M_1(S_{n,m}^1)&\leq k'(n-k')^2+(n-k')(n-1)^2. \label{eq:M1int2}
\end{align}
As $(n,m)$ is in $J$, Lemma~\ref{lemma:kandkprime} gives that $\frac{\sqrt{2}}{2}n-2<k'<\frac{\sqrt{2}}{2}n+1$. Then, $n-k'<(1-\frac{\sqrt{2}}{2})n+2$. 
Replacing these inequalities into Expression~\eqref{eq:M1int2} yields
\begin{align*}
M_1(S_{n,m}^1) &< (\frac{\sqrt{2}}{2}n+1) ((1-\frac{\sqrt{2}}{2})n+2)^2+ ((1-\frac{\sqrt{2}}{2})n+2)(n-1)^2\\
&= (\frac{\sqrt{2}}{2}n+1)( \frac{3-2\sqrt{2}}{2}n^2+(4-2\sqrt{2})n+4)
+ 
 ( (1-\frac{\sqrt{2}}{2})n+2)(n^2-2n+1)\\
&= \frac{\sqrt{2}}{4}n^3+(2\sqrt{2}-\frac{1}{2})n^2+(1-\frac{\sqrt{2}}{2})n+6 = p_2(n)
\end{align*}

Now, let us find a lower bound for $h(n,m)$. By Lemma~\ref{lemma:Hsum} we know the following closed form that holds for any pair $(n,m)$ in $I$,
\begin{equation}\label{eq:hnm2}
h(n,m) = 2m^2-6\binom{n}{3}+(n-1)^2\binom{n}{2}-3(n-1)(n-3)m.    
\end{equation}
As $(n,m)$ is in $J$ we know that $\frac{1}{2}\binom{n}{2}-\frac{n}{2}\leq  m \leq \frac{1}{2}\binom{n}{2}+\frac{n}{2}$. Replacing these bounds for $m$ into Equation~\eqref{eq:hnm2} gives that
\begin{align*}
h(n,m) &\leq 2(\frac{1}{2}\binom{n}{2}+\frac{n}{2})^2-6\binom{n}{3}+(n-1)^2\binom{n}{2}-3(n-1)(n-3)(\frac{1}{2}\binom{n}{2}-\frac{n}{2})\\
&= 2 \left(\frac{n(n+1)}{4}\right)^2-n(n-1)(n-2)+\frac{1}{2}n(n-1)^3-\frac{3}{4}n(n-1)(n-3)^2\\
&= \frac{1}{8}(n^4+2n^3+n^2)-(n^3-3n^2+2n)\\
&+\frac{1}{2}(n^4-3n^3+3n^2-n) - \frac{3}{4}(n^2-n)(n^2-6n+9)\\
&= \frac{1}{8}(-n^4+24n^3-53n^2+34n) = p_3(n). 
\end{align*}
Finally, Expression~\eqref{eq:ineqH} gives that
\begin{align*}
H(C_{n,m}^1)-H(S_{n,m}^1)&\geq 2p_1(n)- (n-\frac{9}{2})p_2(n)-p_3(n)\\
&= \frac{2}{8}(n^4-\frac{15}{8}\sqrt{2}n^3-\frac{1}{2}n^2-13\sqrt{2}n-9)\\
&-(n-\frac{9}{2})(\frac{\sqrt{2}}{4}n^3+(2\sqrt{2}-\frac{1}{2})n^2+(1-\frac{\sqrt{2}}{2})n+6)\\
&-\frac{1}{8}(-n^4+24n^3-53n^2+34n)\\
&= \frac{1}{8}((3-2\sqrt{2})n^4-(20+37\sqrt{2})n^3\\
&+ (19+76\sqrt{2})n^2-(46+226\sqrt{2})n+72) = p(n),
\end{align*}
thus proving Assertion~\ref{i}.
\item As $(n,m)$ is in $J$ we know that $n\geq 8$ and  
$m\leq \frac{1}{2}\binom{n}{2}+\frac{n}{2}$. 
Equation~\eqref{eq:kj'} gives that $m=\binom{n}{2}-\binom{k'+1}{2}+j'$ where $1\leq j'\leq k'$. Consequently, 
$\binom{k'+1}{2}=\binom{n}{2}-m+j'\geq \binom{n}{2}-m\geq \frac{1}{2}\binom{n}{2}-\frac{n}{2}=\frac{n(n-3)}{4}\geq 10$ and, as $k$ is an integer, $k'\geq 4$. In particular, $k'-\frac{7}{2}$ is positive. By Equations~\eqref{eqS2} and \eqref{eqS3}, $H(S_{n,m}^2)-H(S_{n,m}^1)=(k'-\frac{7}{2})(k'-j')(k'-j'-1) \leq (k'-\frac{7}{2})(k')^2$ and 
$H(S_{n,m}^3) - H(S_{n,m}^1)=-3$. Define $g(k')$ as $(k'-\frac{7}{2})(k')^2$. Observe that $g(k')$ is a positive upper bound for $H(S_{n,m}^2) - H(S_{n,m}^1)$. 
Now, let $i$ and $j$ be any integers in $\{1,2,3\}$. If $i=j$ then 
$|H(S_{n,m}^i) - H(S_{n,m}^j)|=0<g(k')<g(k')+3$. If $i\neq j$ but one of them equals $1$, say $i$, then $|H(S_{n,m}^i) - H(S_{n,m}^j)|=|H(S_{n,m}^1) - H(S_{n,m}^j)|\leq \max\{3,g(k')\}\leq g(k')+3$. Otherwise $\{i,j\}=\{2,3\}$, and 
$|H(S_{n,m}^i) - H(S_{n,m}^j)|=|H(S_{n,m}^2) - H(S_{n,m}^3)|=|H(S_{n,m}^2) - H(S_{n,m}^1)|+|H(S_{n,m}^1) - H(S_{n,m}^3)| \leq g(k')+3$. In either case, 
$|H(S_{n,m}^i) - H(S_{n,m}^j)|\leq g(k')+3$. Finally, we will prove that $g(k')+3\leq q(n)$ where $q(n)$ is defined by Equation~\eqref{eq:q}. In fact, by Lemma~\ref{lemma:kandkprime} we know that $k'<\frac{n}{\sqrt{2}}+1$. Consequently,
\begin{align*}
g(k')+3 & = (k'-\frac{7}{2})(k')^2+3 \leq \left(\frac{n}{\sqrt{2}}-\frac{5}{2}\right)
\left( \frac{n}{\sqrt{2}}+1\right)^2 + 3\\
&=  \left(\frac{n}{\sqrt{2}}-\frac{5}{2}\right) 
\left( \frac{n^2}{2}+ \sqrt{2}n+1\right) + 3\\
&= \frac{\sqrt{2}}{4}n^3-\frac{1}{4}n^2-2\sqrt{2}n+\frac{1}{2}\\
&= \frac{1}{8}\left( (2\sqrt{2})x^3-2x^2-16\sqrt{2}x+4\right) = q(n).
\end{align*}
Consequently, we proved that for each pair of integers $i$ and $j$ in $\{1,2,3\}$ it holds that $|H(S_{n,m}^i)-H(S_{n,m}^j)| \leq g(k')+3 \leq q(n)$, as required. \qed\\
\end{enumerate}

\emph{Proof of Lemma~\ref{lemma:437}}: let $n$ and $m$ be integers such that $(n,m)$ is in $I^*$ and $n\geq 437$. 
Define the polynomials $p(x)$ and $q(x)$ by Equations~\eqref{eq:p} and \eqref{eq:q}, respectively. 
Let $r(x)=p(x)-q(x)$. 
As $n\geq 437$ we know by Remark~\ref{remark:r} that $r(n)>0$. Let us prove each of the assertions separately.
\begin{enumerate}[label=(\roman*)]
\item\label{iv} Assume $C_{n,m}^3$ exists. Let $G$ be any graph in $\mathcal{O}_{n,m}-\{C_{n,m}^3\}$. If $G$ is either $C_{n,m}^1$ or $C_{n,m}^2$ then by Lemma~\ref{lemma:chains}\ref{chains2} we know that $H(C_{n,m}^3)>H(G)$. 
Finally, if $G$ is $S_{n,m}^i$ for some $i\in \{1,2,3\}$ then
\begin{align*}
H(C_{n,m}^3)-H(G) &= H(C_{n,m}^3)-H(S_{n,m}^i)\\
&= H(C_{n,m}^3)-H(C_{n,m}^1)+H(C_{n,m}^1)-H(S_{n,m}^1)-
(H(S_{n,m}^i)-H(S_{n,m}^1))\\
&\geq 
H(C_{n,m}^3)-H(C_{n,m}^1)+H(C_{n,m}^1)-H(S_{n,m}^1)-
|H(S_{n,m}^i)-H(S_{n,m}^1)|\\
&\geq  3 + p(n)-q(n) = 3+r(n)>0,
\end{align*}
where we used that $H(C_{n,m}^3)-H(C_{n,m}^1)= 3$ by Equation~\eqref{eqC3}, that $H(C_{n,m}^1)-H(S_{n,m}^1)\geq p(n)$ by Lemma~\ref{lemma:bounds}\ref{i}, and that 
$|H(S_{n,m}^i)-H(S_{n,m}^1)|\leq q(n)$ by Lemma~\ref{lemma:bounds}\ref{ii}. 
\item Assume $C_{n,m}^3$ does not exist. Let $G$ be any graph in $\mathcal{O}_{n,m}-\{C_{n,m}^1\}$. We want to prove that $H(C_{n,m}^1)>H(G)$. If $G$ is $C_{n,m}^2$ then by Lemma~\ref{lemma:chains}\ref{chains2} we know that $H(C_{n,m}^1)>H(G)$. 
Finally, if $G$ is $S_{n,m}^i$ for some $i\in \{1,2,3\}$ then
\begin{align*}
H(C_{n,m}^1)-H(G) &= H(C_{n,m}^1)-H(S_{n,m}^i)=
H(C_{n,m}^1)-H(S_{n,m}^1)-
(H(S_{n,m}^i)-H(S_{n,m}^1))\\
&\geq  p(n)-q(n) = r(n)>0,
\end{align*}
where we used that $H(S_{n,m}^1)-H(C_{n,m}^1)\geq p(n)$ by Lemma~\ref{lemma:bounds}\ref{i}  and the fact that 
$H(S_{n,m}^i)-H(S_{n,m}^1)\leq |H(S_{n,m}^i)-H(S_{n,m}^1)|\leq q(n)$ by Lemma~\ref{lemma:bounds}\ref{ii}. \qed
\end{enumerate}
\end{document}